\documentclass[a4paper,11pt]{article}
\usepackage{amssymb}
\usepackage{amsfonts}
\usepackage{amsmath}

\newtheorem{theorem}{Theorem}[section]

\newtheorem{example}[theorem]{Example}

\newtheorem{proposition}[theorem]{Proposition}

\newenvironment{proof}[1][Proof]{\noindent\textbf{#1.} }{\ \rule{0.5em}{0.5em}}
\input{tcilatex}

\begin{document}

\title{Rotation number values at a discontinuity}
\author{Ricardo Coutinho\  \\
%EndAName
Grupo de F\'{\i}sica Matem\'{a}tica, Universidade de Lisboa.\\
Departamento de Matem\'{a}tica, Instituto Superior T\'{e}cnico, Universidade
de Lisboa,\\
Av. Rovisco Pais, 1049-001 Lisboa, Portugal.\\
Ricardo.Coutinho@tecnico.ulisboa.pt}
\maketitle

\begin{abstract}
In the space of orientation-preserving circle maps that are not necessarily
surjective nor injective, the rotation number does not vary continuously.
Each map where one of these discontinuities occurs is itself discontinuous
and we can consider the possible values of the rotation number when we
modify this map only at its discontinuities. These values are always
rational numbers that necessarily obey a certain arithmetic relation. In
this paper we show that in several examples this relation totally
characterizes the possible values of the rotation number on its
discontinuities, but we also prove that in certain circumstances this
relation is not sufficient for this characterization.
\end{abstract}

\section{Introduction, notations, examples and result}

We shall consider the space $\mathcal{M}$ of lifts of orientation-preserving
circle maps, that is, the set of functions $f:\mathbb{R\longmapsto R}$ which
satisfy the conditions\ 
\begin{equation*}
y>0\Rightarrow f\left( x+y\right) \geqslant f\left( x\right) \quad \text{%
\textrm{and}}\quad f\left( x+1\right) =f\left( x\right) +1.
\end{equation*}%
It should be noted that $\mathcal{M}$ contains functions which are not
continuous or strictly increasing (not surjective nor injective). For each $%
f\in \mathcal{M}\mathbf{\ }$ the limit (rotation number of $f$)%
\begin{equation*}
\nu \left( f\right) =\lim_{n\rightarrow \infty }\frac{f^{n}\left( x\right) }{%
n}
\end{equation*}%
exists and is independent of $x\in \mathbb{R}$ \cite{RhodesThompsom1}.

In $\mathcal{M}$ the rotation number is an increasing functional, $%
f\leqslant g\ \Rightarrow \ \nu \left( f\right) \leqslant \nu \left(
g\right) $, and we may have $\nu \left( f^{-}\right) <\nu \left(
f^{+}\right) $, where 
\begin{equation*}
f^{-}\left( x\right) =\lim_{\QATOP{\delta \rightarrow 0}{\delta >0}}f\left(
x-\delta \right) \quad \text{\textrm{and}}\quad f^{+}\left( x\right) =\lim_{%
\QATOP{\delta \rightarrow 0}{\delta >0}}f\left( x+\delta \right) .
\end{equation*}%
On this space $\mathcal{M}$ we shall consider the L\'{e}vy distance:\textbf{%
\ }%
\begin{equation*}
d_{H}(f,g)=\inf \{\varepsilon >0\ :\ f(x-\varepsilon )-\varepsilon \leqslant
g(x)\leqslant f(x+\varepsilon )+\varepsilon ,\quad \forall x\in \mathbb{R}\
\}.
\end{equation*}%
Observe that we may have $d_{H}(f,g)=0$ with $f\neq g$. In fact it is easy
to verify that 
\begin{equation*}
d_{H}(f,g)=0\quad \Leftrightarrow \quad f^{+}=g^{+}\quad \Leftrightarrow
\quad f^{-}=g^{-}.
\end{equation*}%
Endowed with this distance $\mathcal{M}$ is a pseudometric space. The
following theorem is known.

\begin{theorem}[ \protect\cite{RhodesThompsom2, Coutinho}]
\label{ThContRotNumb}Let $f_{0}\in\mathcal{M}$. Then for every $%
\varepsilon>0 $ there exists $\delta>0$ such that for any $f\in\mathcal{M}$
satisfying $d_{H}(f,f_{0})<\delta$ we have 
\begin{equation*}
\nu\left( f_{0}^{-}\right) -\varepsilon\leqslant\nu\left( f\right)
\leqslant\nu\left( f_{0}^{+}\right) +\varepsilon.
\end{equation*}
\end{theorem}

Therefore, the set of discontinuities of the rotation number is%
\begin{equation*}
\mathcal{D}\equiv\left\{ f\in\mathcal{M}:\ \nu\left( f^{-}\right) <\nu\left(
f^{+}\right) \ \right\} .
\end{equation*}

In \cite{Coutinho} it is shown that if $f\in \mathcal{D}$, then there exists 
$m\in \mathbb{Z}^{+}$ such that $f^{m}$ is a step function, that is to say,
the image $f^{m}\left( \left[ 0,1\right] \right) $ is a finite set. As a
direct consequence of this fact we have that if $f\in \mathcal{M}$ is
continuous or strictly increasing, then $f\notin \mathcal{D}$. Observe,
however, that $f\in \mathcal{D}$ itself does not have to be a step function,
see Examples \ref{ExemploSimples2} and \ref{Exemp2Valores}.

On the other hand Theorem \ref{ThContRotNumb} cannot be improved as the
following proposition shows.

\begin{proposition}
Given $f\in \mathcal{D}$, $\delta >0$ and $\nu \in \mathbb{R}$ satisfying $%
\nu \left( f^{-}\right) <\nu <\nu \left( f^{+}\right) $, there is a
homeomorphism $g\in \mathcal{M}$ such that 
\begin{equation*}
d_{H}(g,f)<\delta \quad \text{\textrm{and}}\quad \nu \left( g\right) =\nu .
\end{equation*}
\end{proposition}

\begin{proof}
For each $\delta >0$ there exist homeomorphisms $h_{-\delta }$ and $%
h_{\delta }\in \mathcal{M}$ such that 
\begin{equation*}
h_{-\delta }\leqslant f\leqslant h_{\delta }\quad \text{\textrm{and}}\quad
d_{H}(h_{\pm \delta },f)<\delta .
\end{equation*}%
We can then construct a family of homeomorphisms $g_{\lambda }\in \mathcal{M}
$, with $\lambda \in \left[ 0,1\right] $, by the formula%
\begin{equation*}
g_{\lambda }=\left( 1-\lambda \right) h_{-\delta }+\lambda h_{\delta }.
\end{equation*}%
Therefore $d_{H}(g_{\lambda },f)<\delta $, for all $\lambda \in \left[ 0,1%
\right] $, and $\nu \left( g_{0}\right) =\nu \left( h_{-\delta }\right)
\leqslant \nu \left( f^{-}\right) <\nu \left( f^{+}\right) \leqslant \nu
\left( h_{\delta }\right) =\nu \left( g_{1}\right) $. Since the rotation
number is continuous in the subspace of the homeomorphisms of $\mathcal{M}$,
we have that, if $\nu \left( f^{-}\right) <\nu <\nu \left( f^{+}\right) $,
then there exists $\lambda _{0}\in \left[ 0,1\right] \ $such that $\nu
\left( g_{\lambda _{0}}\right) =\nu $.
\end{proof}

Although in an arbitrary neighborhood of $f\in \mathcal{D}$ we have an
interval of possible values of the rotation number, the same is not true if
we consider only the functions that are at a null distance from $f$. In \cite%
{Coutinho} is given a characterization of the possible rotation numbers for
functions in these circumstances:

\begin{theorem}[\protect\cite{Coutinho}]
\label{TeorNumRotDesc}Let $f_{0},f_{1}\in\mathcal{D}$ be such that $%
d_{H}(f_{0},f_{1})=0$ and $\nu\left( f_{0}\right) <\nu\left( f_{1}\right) $,
then $\nu\left( f_{0}\right) =\frac{p_{0}}{q_{0}}$ and $\nu\left(
f_{1}\right) =\frac{p_{1}}{q_{1}}$ are rationals that, when represented as
irreducible fractions, satisfy the condition 
\begin{equation*}
\frac{p_{1}-1}{q_{1}}\leqslant\frac{p_{0}}{q_{0}}<\dfrac{p_{1}}{q_{1}}%
\leqslant\frac{p_{0}+1}{q_{0}}.
\end{equation*}
\end{theorem}

In particular, if we know the values of $\nu \left( f^{-}\right) $ and $\nu
\left( f^{+}\right) $ we have only a finite set of possible values for $\nu
\left( f\right) $. In the next proposition we give a (non-injective)
parameterization of this set where we use the floor and the ceiling integer
functions respectively defined by 
\begin{equation*}
\left\lfloor x\right\rfloor =\max \left\{ n\in \mathbb{Z}:\mathbb{\quad }%
n\leqslant x\right\} \quad \text{\textrm{and}}\quad \left\lceil
x\right\rceil =\min \left\{ n\in \mathbb{Z}:\mathbb{\quad }x\leqslant
n\right\} .
\end{equation*}

\begin{proposition}
\label{PropoS}If $f\in\mathcal{D}$ is such that $\nu^{-}\equiv$ $\nu\left(
f^{-}\right) <\nu\left( f\right) <\nu\left( f^{+}\right) \equiv\nu^{+}$,
then $\nu\left( f\right) $ belongs to the following finite set 
\begin{equation*}
\nu\left( f\right) \in S_{\nu^{-},\nu^{+}}\equiv\left\{ \frac{\left\lfloor
q\,\nu^{-}\right\rfloor +1}{q}:\quad\left\lceil q\,\nu^{+}\right\rceil
=\left\lfloor q\,\nu^{-}\right\rfloor +2,\quad q\in\mathbb{Z}^{+}\right\} .
\end{equation*}
\end{proposition}

\begin{proof}
If $\nu\left( f\right) =\frac{p}{q}$ (irreducible fraction), we know from
Theorem \ref{TeorNumRotDesc} (applied to pairs $\left( f^{-},f\right) $ and $%
\left( f,f^{+}\right) $) that we have, with $\frac{p_{-}}{q_{-}}=\nu^{-}$
and $\frac{p_{+}}{q_{+}}=\nu^{+}$, 
\begin{equation*}
p-1\leqslant q\dfrac{p_{-}}{q_{-}}<p<q\dfrac{p_{+}}{q_{+}}\leqslant p+1.
\end{equation*}
Equivalently $p=\left\lfloor q\,\nu^{-}\right\rfloor +1=\left\lceil q\,\nu
^{+}\right\rceil -1$ and the condition $\left\lceil q\,\nu^{+}\right\rceil
=\left\lfloor q\,\nu^{-}\right\rfloor +2$ can only be true for a finite
number of values of $q\in\mathbb{Z}^{+}$ since $\,\nu^{-}<\nu^{+}$ (in fact
we can even show that $\left( q_{-}+q_{+}\right) /\Delta\leqslant q\leqslant
2q_{-}q_{+}/\Delta$, where $\Delta=p_{+}q_{-}-p_{-}q_{+})$.
\end{proof}

The purpose of this article is to evaluate the extent to which Theorem \ref%
{TeorNumRotDesc} is insightful in describing the set of rotation number
values at a discontinuity $f\in \mathcal{D}$ that we can define symbolically
by%
\begin{equation*}
V\left( f\right) \equiv \left\{ \nu \left( g\right) :\ g\in \mathcal{M}\ 
\text{\textrm{and}}\ d_{H}(f,g)=0\right\} .
\end{equation*}%
With this notation Theorem \ref{TeorNumRotDesc} states%
\begin{equation}
V\left( f\right) \subset \left\{ \nu ^{-},\nu ^{+}\right\} \cup S_{\nu
^{-},\nu ^{+}},  \label{RelInclusV(f)<S}
\end{equation}%
with $\nu ^{-}=$ $\nu \left( f^{-}\right) $, $\nu ^{+}=\nu \left(
f^{+}\right) $ and $S_{\nu ^{-},\nu ^{+}}$ defined in Proposition \ref%
{PropoS}. The question that arises is whether we can replace the inclusion
by an equality. We will see in the following examples that the answer may be
affirmative, but it may be negative as well. Let us start by looking at an
example where $V\left( f\right) =\left\{ \nu ^{-},\nu ^{+}\right\} \cup
S_{\nu ^{-},\nu ^{+}}$.

\begin{example}
\label{ExemploSimples}Let $f=\frac{\left\lceil 2x\right\rceil }{2}$, $g_{1}=%
\frac{1+\left\lceil 2x\right\rceil +\left\lfloor 2x\right\rfloor }{4}$and $%
g_{2}=\frac{1+\left\lceil x\right\rceil +\left\lfloor 2x\right\rfloor
+\left\lfloor x+\frac{1}{2}\right\rfloor }{4}$ (see Figure \ref{Fig1}). 
\FRAME{ftbhFU}{15.328cm}{3.8836cm}{0pt}{\Qcb{Graph of $f$, $g_{1}$, $g_{2}$
and $f^{+}$ from Example \protect\ref{ExemploSimples}. Their rotation
numbers are $0$, $1/4$, $1/3$ and $1/2$, respectively.}}{\Qlb{Fig1}}{%
grafseclin10.gif}{\special{language "Scientific Word";type
"GRAPHIC";maintain-aspect-ratio TRUE;display "USEDEF";valid_file "F";width
15.328cm;height 3.8836cm;depth 0pt;original-width 37.417in;original-height
9.3538in;cropleft "0";croptop "1";cropright "1.0001";cropbottom "0";filename
'Figuras/GrafSecLin10.gif';file-properties "XNPEU";}} We have $f^{+}=\frac{%
1+\left\lfloor 2x\right\rfloor }{2}$, $f^{-}=g_{1}^{-}=g_{2}^{-}=f\ $and$\
g_{1}^{+}=g_{2}^{+}=f^{+}$. Since $f\left( 0\right) =0$, $g_{1}^{4}\left(
0\right) =1$, $g_{2}^{3}\left( 0\right) =1$, $\left. f^{+}\right. ^{2}\left(
0\right) =1$, we obtain $\nu \left( f\right) =0$, $\nu \left( g_{1}\right) =%
\frac{1}{4}$, $\nu \left( g_{2}\right) =\frac{1}{3}$ and $\nu \left(
f^{+}\right) =\frac{1}{2}$. On the other hand $S_{0,\frac{1}{2}}=\left\{ 
\frac{1}{4},\frac{1}{3}\right\} $, therefore $V\left( f\right) =\left\{ 0,%
\frac{1}{2}\right\} \cup S_{0,\frac{1}{2}}$.
\end{example}

The next example shows that we can also have $V\left( f\right) \neq\left\{
\nu^{-},\nu^{+}\right\} \cup S_{\nu^{-},\nu^{+}}$.

\begin{example}
\label{ExemploSimples2}Let $f\left( x\right) =\min \left( x+\frac{1}{2}%
,\left\lceil x\right\rceil \right) $ (see Figure \ref{Fig2}). \FRAME{ftbhFU}{%
11.5125cm}{3.8858cm}{0pt}{\Qcb{Graph of functions $f\left( x\right) $, $g=%
\frac{1}{2}\left( f+f^{+}\right) $ and $f^{+}$ from Example \protect\ref%
{ExemploSimples2}. Their rotation numbers are $0,1/3$ and $1/2$,
respectively.}}{\Qlb{Fig2}}{grafseclin20.gif}{\special{language "Scientific
Word";type "GRAPHIC";maintain-aspect-ratio TRUE;display "USEDEF";valid_file
"F";width 11.5125cm;height 3.8858cm;depth 0pt;original-width
28.0623in;original-height 9.3538in;cropleft "0";croptop "1";cropright
"1";cropbottom "0";filename 'Figuras/GrafSecLin20.gif';file-properties
"XNPEU";}} We have $f^{-}=$ $f$, $f\left( 0\right) =0$, $\left. f^{+}\right.
^{2}\left( 0\right) =f^{+}\left( \frac{1}{2}\right) =1$, therefore $\nu
\left( f^{-}\right) =0$ and $\nu \left( f^{+}\right) =\frac{1}{2}$. But if $%
g\in \mathcal{M}$ is such that $d_{H}(f,g)=0$ and $f^{-}\neq g\neq f^{+}$,
then $0<g\left( 0\right) <\frac{1}{2}$ and $g^{3}\left( 0\right) =$ $g\left(
g\left( 0\right) +\frac{1}{2}\right) =1$; so that $\nu \left( g\right) =%
\frac{1}{3}$. Hence $V\left( f\right) =\left\{ 0,\frac{1}{3},\frac{1}{2}%
\right\} \neq \left\{ 0,\frac{1}{2}\right\} \cup S_{0,\frac{1}{2}}=\left\{ 0,%
\frac{1}{4},\frac{1}{3},\frac{1}{2}\right\} $.
\end{example}

The preceding examples are very particular (for being simple) and suggest
several conjectures that are not true; so it is convenient to give two less
trivial examples.

\begin{example}
\label{ExemploComplicado}Let$f\left( x\right) $ be defined by the expression 
\begin{equation*}
\frac{1}{10}\left( 4+2\left\lceil x\right\rceil +\left\lceil x-\tfrac{1}{10}%
\right\rceil +\left\lceil x-\tfrac{1}{5}\right\rceil +2\left\lceil x-\tfrac{2%
}{5}\right\rceil +\left\lceil x-\tfrac{1}{2}\right\rceil +\left\lceil x-%
\tfrac{3}{5}\right\rceil +2\left\lceil x-\tfrac{4}{5}\right\rceil \right)
\end{equation*}%
and $f_{1}$, $f_{2}$, $f_{3}$ and $f_{4}$ according to Figure \ref{Fig3}. 
\FRAME{ftbhFU}{15.4269cm}{4.4613cm}{0pt}{\Qcb{Graph of the functions $f_{1}$%
, $f_{2}$, $f_{3}$ and $f_{4}$ from Example \protect\ref{ExemploComplicado}.
Their rotation numbers are $2/7$, $3/10$, $1/3$ and $3/8$, respectively.}}{%
\Qlb{Fig3}}{grafseclin30.gif}{\special{language "Scientific Word";type
"GRAPHIC";maintain-aspect-ratio TRUE;display "USEDEF";valid_file "F";width
15.4269cm;height 4.4613cm;depth 0pt;original-width 37.5in;original-height
10.7185in;cropleft "0";croptop "1";cropright "1";cropbottom "0";filename
'Figuras/GrafSecLin30.gif';file-properties "XNPEU";}} Then ( $s=1,2,3,4$)

\begin{equation*}
f^{-}=f_{s}^{-}=f\quad \mathrm{and}\quad f_{s}^{+}=f^{+}.
\end{equation*}%
By calculating the successive iterates of the point $0$ by each of these
functions we obtain 
\begin{equation*}
\nu \left( f\right) =\tfrac{1}{4}\ ,\quad \nu \left( f_{1}\right) =\tfrac{2}{%
7}\ ,\quad \nu \left( f_{2}\right) =\tfrac{3}{10}\ ,\quad \nu \left(
f_{3}\right) =\tfrac{1}{3}\ ,\quad \nu \left( f_{4}\right) =\tfrac{3}{8}\
,\quad \nu \left( f^{+}\right) =\tfrac{2}{5}.
\end{equation*}%
On the other hand $S_{\frac{1}{4},\frac{2}{5}}=\left\{ \frac{2}{7},\frac{3}{%
10},\frac{1}{3},\frac{3}{8}\right\} $, so in this case $V\left( f\right)
=\left\{ \frac{1}{4},\frac{2}{5}\right\} \cup S_{\frac{1}{4},\frac{2}{5}}$.
\end{example}

The next example also shows that, in general, 
\begin{equation*}
V\left( f\right) \neq \left\{ \nu \left( \left( 1-\lambda \right)
f^{-}+\lambda f^{+}\right) :\lambda \in \left[ 0,1\right] \right\} .
\end{equation*}

\begin{example}
\label{Exemp2Valores}Given $\left( \alpha ,\beta \right) \in \left[ 0,1%
\right] ^{2},$ let $f_{\alpha ,\beta }\in \mathcal{M}$ be defined on $\left[
0,1\right) $ by 
\begin{equation*}
f_{\alpha ,\beta }\left( x\right) =\left\{ 
\begin{array}{ccl}
\left( 1+2\alpha \right) /6 & \mathrm{if} & x=0 \\ 
{\small 1/2} & \mathrm{if} & 0<x<1/3 \\ 
\left( 1+\beta \right) /2 & \mathrm{if} & x=1/3 \\ 
{\small 1} & \mathrm{if} & 1/3<x\leqslant 5/6 \\ 
x+1/6 & \mathrm{if} & 5/6\leqslant x<1%
\end{array}%
\right.
\end{equation*}%
and by $f_{\alpha ,\beta }\left( x\right) =f_{\alpha ,\beta }\left(
x-\left\lfloor x\right\rfloor \right) +\left\lfloor x\right\rfloor $ on the
remaining points (see Figure \ref{Fig4}). \FRAME{ftbhFU}{11.5279cm}{4.4372cm%
}{0pt}{\Qcb{Graph of the function $f_{\protect\alpha ,\protect\beta }\left(
x\right) $ with $\protect\alpha \in \left\{ \frac{1}{4},\frac{1}{2},\frac{3}{%
4}\right\} $ and $\protect\beta =\frac{5}{6}$. Their rotation numbers are $%
1/3$, $2/5$ and $1/2$, respectively.}}{\Qlb{Fig4}}{grafseclin40.gif}{\special%
{language "Scientific Word";type "GRAPHIC";maintain-aspect-ratio
TRUE;display "USEDEF";valid_file "F";width 11.5279cm;height 4.4372cm;depth
0pt;original-width 28.1254in;original-height 10.7185in;cropleft "0";croptop
"1";cropright "1";cropbottom "0";filename
'Figuras/GrafSecLin40.gif';file-properties "XNPEU";}} By calculating the
successive iterates of the point $0$ we obtain 
\begin{equation*}
\nu \left( f_{\alpha ,\beta }\right) =\left\{ 
\begin{array}{ccl}
{\small 1/3} & \mathrm{if} & \alpha <{\small 1/2}\ \mathrm{or}\ (\alpha =%
{\small 1/2}\ \mathrm{and}\ \beta \leqslant {\small 2/3}) \\ 
{\small 2/5} & \mathrm{if} & \alpha ={\small 1/2}\ \mathrm{and}\ {\small 2/3}%
<\beta <1 \\ 
{\small 1/2} & \mathrm{if} & (\alpha ={\small 1/2}\ \mathrm{and}\ \beta =1)\ 
\mathrm{or}\ {\small 1/2}<\alpha%
\end{array}%
\right. .
\end{equation*}

Since $d_{H}(f,f_{0,0})=0$ if and only if $f=f_{\alpha,\beta}$ for some $%
\left( \alpha,\beta\right) \in\left[ 0,1\right] ^{2}$, and in this case $%
f^{-}=f_{0,0}$ and $f^{+}=f_{1,1}$; we deduce $V\left( f_{\alpha,\beta
}\right) =\left\{ \frac{1}{3},\frac{2}{5},\frac{1}{2}\right\} $, however 
\begin{equation*}
S_{\frac{1}{3},\frac{1}{2}}=\left\{ \tfrac{3}{8},\tfrac{2}{5},\tfrac{5}{12},%
\tfrac{3}{7},\tfrac{4}{9}\right\} .
\end{equation*}

Note also that for $\lambda \in \left[ 0,1\right] $ (and $f=f_{\alpha ,\beta
}$) one has $\left( 1-\lambda \right) f^{-}+\lambda f^{+}=f_{\lambda
,\lambda }$ so that 
\begin{equation*}
\left\{ \nu \left( \left( 1-\lambda \right) f^{-}+\lambda f^{+}\right)
:\lambda \in \left[ 0,1\right] \right\} =\left\{ \tfrac{1}{3},\tfrac{1}{2}%
\right\} \neq V\left( f\right) .
\end{equation*}
\end{example}

Another question that arises is whether in inclusion (\ref{RelInclusV(f)<S})
we can replace $S_{\nu ^{-},\nu ^{+}}$ by a smaller set. By defining 
\begin{equation*}
V_{\nu ^{-},\nu ^{+}}\equiv \left\{ \nu \left( f\right) :\ f\in \mathcal{M}%
\quad \text{\textrm{and}}\quad \nu ^{-}=\text{ }\nu \left( f^{-}\right) <\nu
\left( f\right) <\nu \left( f^{+}\right) =\nu ^{+}\right\} ,
\end{equation*}%
we obviously have $V\left( f\right) \subset \left\{ \nu ^{-},\nu
^{+}\right\} \cup V_{\nu ^{-},\nu ^{+}}$ if $\nu ^{-}=$ $\nu \left(
f^{-}\right) $ and $\nu ^{+}=\nu \left( f^{+}\right) $; and therefore the
problem is whether $V_{\nu ^{-},\nu ^{+}}=S_{\nu ^{-},\nu ^{+}}$? This
equality is verified in certain cases as shown in Examples \ref%
{ExemploSimples} and \ref{ExemploComplicado}, but surprisingly it is not
true in general as shown by the following theorem which we shall prove in
Section \ref{SecDem}.

\begin{theorem}
\label{TeorPrincipal}If $\frac{p}{q}\in S_{\nu^{-},\nu^{+}}$ is such that $%
\frac{p}{q}$ is irreducible, $\frac{p-1}{q}=\nu^{-}$, $\frac{p+1}{q}=\nu^{+}$
and $q$ is odd, then there is no $f\in\mathcal{M}$ with $\nu\left(
f^{-}\right) =\nu^{-}$, $\nu\left( f^{+}\right) =\nu^{+}$ and $\nu\left(
f\right) =\frac{p}{q}$.
\end{theorem}

To be sure that the previous theorem is relevant we need an example of a
function $f\in \mathcal{D}$ such that for a certain irreducible fraction $%
\frac{p}{q}$ with odd $q$, we have $\frac{p-1}{q}=\nu \left( f^{-}\right) $
and $\frac{p+1}{q}=\nu \left( f^{+}\right) $; which we shall see in the next
example. In fact it would have been enough to give examples of irreducible
fractions $\frac{p_{-}}{q_{-}},\frac{p}{q},\frac{p_{+}}{q_{+}}$ such that 
\begin{equation*}
\frac{p_{+}-1}{q_{+}}\leqslant \frac{p_{-}}{q_{-}}=\frac{p-1}{q}\quad \text{%
\textrm{and}}\quad \frac{p+1}{q}=\frac{p_{+}}{q_{+}}\leqslant \frac{p_{-}+1}{%
q_{-}},
\end{equation*}%
since in \cite{Coutinho} it is shown that if $\frac{p_{+}-1}{q_{+}}\leqslant 
\frac{p_{-}}{q_{-}}<\frac{p_{+}}{q_{+}}\leqslant \frac{p_{-}+1}{q_{-}}$,
then there exists $f\in \mathcal{M}$ with $\nu \left( f^{-}\right) =\tfrac{%
p_{-}}{q_{-}}$ and $\nu \left( f^{+}\right) =\tfrac{p_{+}}{q_{+}}$.

\begin{example}
\label{ExemploFinal}For $f\in \mathcal{D}$ defined by the following
expression (see Figure \ref{Fig5}), \FRAME{ftbhFU}{11.5279cm}{4.4372cm}{0pt}{%
\Qcb{Graph of the functions $f$, $f\left( x\right) +\frac{1}{10}\left(
\left\lfloor 1+x-\frac{1}{5}\right\rfloor -\left\lceil x-\frac{1}{5}%
\right\rceil \right) $ and $f^{+}$ from Example \protect\ref{ExemploFinal}.
Their rotation numbers are $1/5$, $1/4$ and $1/3$, respectively.}}{\Qlb{Fig5}%
}{grafseclin50.gif}{\special{language "Scientific Word";type
"GRAPHIC";maintain-aspect-ratio TRUE;display "USEDEF";valid_file "F";width
11.5279cm;height 4.4372cm;depth 0pt;original-width 28.1254in;original-height
10.7185in;cropleft "0";croptop "1";cropright "1";cropbottom "0";filename
'Figuras/GrafSecLin50.gif';file-properties "XNPEU";}} 
\begin{equation*}
f\left( x\right) =\tfrac{1}{5}\left( 1+\left\lceil 5x\right\rceil
-\left\lceil x\right\rceil +\left\lceil x-\tfrac{1}{10}\right\rceil \right)
\end{equation*}%
it is easy to verify that $\nu \left( f^{-}\right) =\frac{1}{5}$ and $\nu
\left( f^{+}\right) =\frac{1}{3}$. Also $\frac{4-1}{15}=\frac{1}{5}$ and $%
\frac{4+1}{15}=\frac{1}{3}$; so $\frac{4}{15}\in S_{\frac{1}{5},\frac{1}{3}}$
and Theorem \ref{TeorPrincipal} shows that $\frac{4}{15}\notin V_{\frac{1}{5}%
,\frac{1}{3}}$. Therefore $V_{\frac{1}{5},\frac{1}{3}}\neq S_{\frac{1}{5},%
\frac{1}{3}}$.

Although it is easy to see that, for example, $\tfrac{3}{11}\notin V\left(
f\right) $, by constructing other examples $g\in\mathcal{M}$ with $\nu\left(
g^{-}\right) =\frac{1}{5}$ and $\nu\left( g^{+}\right) =\frac{1}{3}$, it is
possible to show that $V_{\frac{1}{5},\frac{1}{3}}=S_{\frac{1}{5},\frac{1}{3}%
}\setminus\left\{ \frac{4}{15}\right\} =\left\{ \tfrac{2}{9},\tfrac{1}{4},%
\tfrac{3}{11},\tfrac{2}{7},\tfrac{3}{10}\right\} $.
\end{example}

\section{Proof of Theorem \protect\ref{TeorPrincipal}\label{SecDem}}

Since each $f\in\mathcal{M}$ represents a map $\varphi:S^{1}\rightarrow
S^{1} $, we define an orbit of $f$ as a set of the form 
\begin{equation*}
\left\{ f^{k}\left( x_{0}\right) +m:\ k\in\mathbb{N}\ ,\ m\in \mathbb{Z}%
\right\} ,
\end{equation*}
where $x_{0}\in\mathbb{R}$. This orbit is periodic if there exist a $p\in%
\mathbb{Z}$ and a $q\in\mathbb{Z}^{+}$ such that $f^{q}\left( x_{0}\right)
=x_{0}+p$.

We know from \cite{Coutinho} that any $f\in \mathcal{D}$ has at least one
periodic orbit. Also that if $f_{0},f_{1}\in \mathcal{D}$ are such that $%
d_{H}(f_{0},f_{1})=0$ with $\nu \left( f_{0}\right) \neq \nu \left(
f_{1}\right) $, then any periodic orbit of $f_{0}$ intersects all periodic
orbits of $f_{1}$ (if this were not the case it would be possible to
construct $g$ by modifying only the discontinuities of $f_{0}$ in such a way
that $g$ maintains the periodic orbit of $f_{0}$ and also has one of $f_{1}$
in contradiction to the uniqueness of the rotation number).

In the proofs bellow, we will mainly use these facts and the following
trivial property:%
\begin{equation*}
\mathrm{if\ }d_{H}(f,g)=0\mathrm{\ and\ }x<y\mathrm{,\ then\ }f\left(
x\right) \leqslant g\left( y\right) .
\end{equation*}

\begin{proposition}
\label{Prop1}Suppose that $f,f_{1}\in\mathcal{D}$ are such that $%
d_{H}(f,f_{1})=0$ and write $\frac{p}{q}=\nu\left( f\right) $ and $\frac{%
p_{1}}{q_{1}}=\nu\left( f_{1}\right) $ as irreducible fractions. If $\frac {%
p+1}{q}=\frac{p_{1}}{q_{1}}$, then every periodic orbit of $f_{1}$ is
contained in a periodic orbit of $f$ and every periodic orbit of $f$
contains a periodic orbit of $f_{1}$.
\end{proposition}

\begin{proof}
Let $x_{0}$ be a point common to a periodic orbit of $f$ and $f_{1}$. So that%
\begin{equation*}
f^{q}\left( x_{0}\right) =x_{0}+p\quad \text{\textrm{and}}\quad
f_{1}^{q_{1}}\left( x_{0}\right) =x_{0}+p_{1}.
\end{equation*}%
Let $x_{j}$ be an increasing enumeration of the periodic orbit of $f$
passing through $x_{0}$, that is, 
\begin{equation*}
\left\{ x_{j}:j\in \mathbb{Z}\right\} =\left\{ f^{k}\left( x_{0}\right) +m:\
k\in \mathbb{N}\ ,\ m\in \mathbb{Z}\right\} ,
\end{equation*}%
with $x_{j}<x_{j+1}$ for all $j\in \mathbb{Z}$. Hence 
\begin{equation*}
x_{j+q}=x_{j}+1\quad \text{\textrm{and}}\quad f\left( x_{j}\right) =x_{j+p}.
\end{equation*}%
We will prove the proposition by showing that for every $k\in \mathbb{N}$ we
have%
\begin{equation}
f_{1}^{k}\left( x_{0}\right) =x_{k\left( p+1\right) }.  \label{RelPro1}
\end{equation}%
Let us first see that this relation (\ref{RelPro1}) is true when $k$ is a
multiple of $qq_{1}$. In fact (using $\frac{p+1}{q}=\frac{p_{1}}{q_{1}}$)%
\begin{equation*}
f_{1}^{nqq_{1}}\left( x_{0}\right) =x_{0}+nqp_{1}=x_{0}+nq_{1}\left(
p+1\right) =x_{nqq_{1}\left( p+1\right) }.
\end{equation*}%
On the other hand, if $x_{k\left( p+1\right) }\leqslant f_{1}^{k}\left(
x_{0}\right) $ is true for some $k\in \mathbb{Z}^{+}$, then 
\begin{equation*}
f\left( x_{\left( k-1\right) \left( p+1\right) }\right) =x_{k\left(
p+1\right) -1}<x_{k\left( p+1\right) }\leqslant f_{1}^{k}\left( x_{0}\right)
=f_{1}\circ f_{1}^{k-1}\left( x_{0}\right)
\end{equation*}%
and therefore we must have $x_{\left( k-1\right) \left( p+1\right)
}\leqslant f_{1}^{k-1}\left( x_{0}\right) $. Then we conclude by descending
induction that for every $k\in \mathbb{N}$ we have%
\begin{equation*}
x_{k\left( p+1\right) }\leqslant f_{1}^{k}\left( x_{0}\right) .
\end{equation*}%
Also, if $f_{1}^{k}\left( x_{0}\right) \leqslant x_{k\left( p+1\right) }$ is
true for some $k\in \mathbb{Z}^{+}$, then, since $x_{k\left( p+1\right)
}<x_{k\left( p+1\right) +1}$, we obtain 
\begin{equation*}
f_{1}^{k+1}\left( x_{0}\right) \leqslant f_{1}\left( x_{k\left( p+1\right)
}\right) \leqslant f\left( x_{k\left( p+1\right) +1}\right) =x_{k\left(
p+1\right) +1+p}=x_{\left( k+1\right) \left( p+1\right) },
\end{equation*}%
which proves by induction that the relation%
\begin{equation*}
f_{1}^{k}\left( x_{0}\right) \leqslant x_{k\left( p+1\right) }
\end{equation*}%
is also true for all $k\in \mathbb{N}$; so that the proposition is proved.
\end{proof}

\begin{proposition}
Suppose that $f_{0},f\in\mathcal{D}$ are such that $d_{H}(f_{0},f)=0$ and
write $\frac{p_{0}}{q_{0}}=\nu\left( f_{0}\right) $ and $\frac{p}{q}%
=\nu\left( f\right) $ as irreducible fractions. If $\frac{p_{0}}{q_{0}}=%
\frac{p-1}{q}$, then every periodic orbit of $f_{0}$ is contained in a
periodic orbit of $f$ and every periodic orbit of $f$ contains a periodic
orbit of $f_{0}$.
\end{proposition}

\begin{proof}
Similarly to the proof of Proposition \ref{Prop1}, using the same notation
for $x_{j}$, where now $x_{0}$ is a point common to a periodic orbit of $f$
and $f_{0}$, we prove successively that $f_{0}^{nqq_{0}}\left( x_{0}\right)
=x_{nqq_{0}\left( p-1\right) }$, that $f_{0}^{k}\left( x_{0}\right)
\leqslant x_{k\left( p-1\right) }$ by descending induction and that $%
x_{k\left( p-1\right) }\leqslant f_{0}^{k}\left( x_{0}\right) $ by usual
induction. We obtain 
\begin{equation}
f_{0}^{k}\left( x_{0}\right) =x_{k\left( p-1\right) }  \label{RelPro2}
\end{equation}
for all $k\in\mathbb{N}$, which was what we wanted to prove.
\end{proof}

\begin{proposition}
\label{Prop3}Suppose that $f\in\mathcal{D}$, $\frac{p}{q}=\nu\left( f\right) 
$ is an irreducible fraction, $\frac{p-1}{q}=\nu\left( f^{-}\right) $ and $%
\frac{p+1}{q}=\nu\left( f^{+}\right) $. Then $p$ is odd.
\end{proposition}

\begin{proof}
Let $x_{0}$ be a point common to a periodic orbit of $f^{-}$ and $f^{+}$. By
the previous propositions, $x_{0}$ belongs to a periodic orbit of $f$ which,
as before, we denote by $\left\{ x_{j}:j\in\mathbb{Z}\right\} $ with $%
x_{j}<x_{j+1}$. Using the relations (\ref{RelPro1}) and (\ref{RelPro2}) we
have for every $k\in\mathbb{N}$ 
\begin{equation*}
\left. f^{-}\right. ^{k}\left( x_{0}\right) =x_{k\left( p-1\right) }\quad%
\text{\textrm{and}}\quad\left. f^{+}\right. ^{k}\left( x_{0}\right)
=x_{k\left( p+1\right) }.
\end{equation*}

Let us first note that $p\neq0$; in fact Theorem \ref{TeorNumRotDesc}
applied to $f^{-}$ and $f^{+}$ shows in particular that $\nu\left(
f^{-}\right) \nu\left( f^{+}\right) \geqslant0$ and then $\nu\left(
f^{-}\right) <\frac{p}{q}<\nu\left( f^{+}\right) $ implies $p\neq0$.

If $p$ were even, then $p-1$ and $p+1$ would be coprime of the same sign, so
that there would exist $k_{0}$ and $k_{1}$ in $\mathbb{N}$ such that 
\begin{equation*}
k_{0}\left( p-1\right) -k_{1}\left( p+1\right) =1.
\end{equation*}%
Hence $x_{k_{1}\left( p+1\right) }<x_{k_{0}\left( p-1\right) }$ and therefore%
\begin{equation*}
x_{\left( k_{1}+1\right) \left( p+1\right) }=f^{+}\left( x_{k_{1}\left(
p+1\right) }\right) \leqslant f^{-}\left( x_{k_{0}\left( p-1\right) }\right)
=x_{\left( k_{0}+1\right) \left( p-1\right) },
\end{equation*}%
in contradiction to 
\begin{equation*}
\left( k_{0}+1\right) \left( p-1\right) -\left( k_{1}+1\right) \left(
p+1\right) =-1<0,
\end{equation*}%
which implies $x_{\left( k_{0}+1\right) \left( p-1\right) }<x_{\left(
k_{1}+1\right) \left( p+1\right) }$. Hence $p$ must be odd.
\end{proof}

We can now easily prove Theorem \ref{TeorPrincipal} by applying Proposition %
\ref{Prop3} to $f$ and $f+1$, which have rotation numbers $\nu \left(
f\right) =\frac{p}{q}$ and $\nu \left( f+1\right) =\frac{p+q}{q}$. We find
that $p$ and $p+q$ are odd, and therefore $q$ is even.

\end{document}